\documentclass[12pt]{amsart}
\usepackage{amscd,amsfonts,amsthm,amssymb,latexsym,amsmath, amscd, amsmath}
\usepackage[all]{xy}
\usepackage{anysize}
\usepackage[sc]{mathpazo} 
\usepackage{xcolor}

\usepackage[francais,english]{babel}
\usepackage[utf8]{inputenc}   
\usepackage{hyperref}
\usepackage{stmaryrd}
\usepackage{fancyhdr}
\usepackage{url}
\usepackage{dsfont}

\marginsize{2.5cm}{2.5cm}{2.5cm}{2.5cm}
\newtheorem{theorem}{Theorem}[section] 
\newtheorem{lemma}[theorem]{Lemma}
\newtheorem{corollary}[theorem]{Corollary}
\newtheorem{proposition}[theorem]{Proposition}
\theoremstyle{definition}

\theoremstyle{remark}
\newtheorem{remark}{Remark}

\renewcommand{\o}{\mathfrak{o}}
\newcommand{\e}{\epsilon}
\newcommand{\p}{\mathfrak{p}}
\newcommand{\w}{\varpi}
\renewcommand{\d}{\delta}

\newcommand{\C}{\mathbb{C}}

\begin{document}

\title[Test vectors for local periods]
{Test vectors for local periods}
\author{U. K. Anandavardhanan and Nadir Matringe}

\address{Department of Mathematics, Indian Institute of Technology Bombay, Mumbai - 400076, India.}
\email{anand@math.iitb.ac.in}

\address{Laboratoire Math\'ematiques et Applications, Universit\'e de Poitiers, France.}
\email{matringe@math.univ-poitiers.fr}


\date{}

\begin{abstract}
Let $E/F$ be a quadratic extension of non-archimedean local fields of characteristic zero. An irreducible admissible representation $\pi$ of $GL(n,E)$ is said to be distinguished with respect to $GL(n,F)$ if it admits a non-trivial linear form that is invariant under the action of $GL(n,F)$. It is known that there is exactly one such invariant linear form up to multiplication by scalars, and an explicit linear form is given by integrating Whittaker functions over the $F$-points of the mirabolic subgroup when $\pi$ is unitary and generic. In this paper, we prove that the essential vector of \cite{jps81} is a test vector for this standard distinguishing linear form and that the value of this form at the essential vector is a local $L$-value. As an application we determine the value of a certain proportionality constant between two explicit distinguishing linear forms. We then extend all our results to the non-unitary generic case.
\end{abstract}

\maketitle

\section{Introduction}

Let $F$ be a non-archimedean local field of characteristic zero. Let $G$ be the $F$-points of a reductive algebraic group over $F$ and let $H$ be the $F$-points of a reductive subgroup of $G$ over $F$. An irreducible representation, say $\pi$, of $G$ is said to be \textit{$\chi$-distinguished} with respect to $H$, where $\chi:H \rightarrow \mathbb C^\times$ is a character of $H$, if it admits a non-trivial $(H,\chi)$-equivariant linear form, which is to say
\[{\rm Hom}_H(\pi,\chi) \neq 0.\] The case when $G$ is quasi-split, $H$ is a maximal unipotent subgroup and $\chi$ is a non-degenerate character of $H$, in which situation the above space is known to be at most one dimensional, gives rise to the \textit{Whittaker model}, an object that is all pervasive and prominent in local harmonic analysis.
 
The representation $\pi$ of $G$ is said to be \textit{distinguished} with respect to $H$ if it is $1$-distinguished with respect to $H$, where $1$ is the trivial representation of $H$. 

Distinguished representations arise naturally in various situations of wide interest. For instance, they play an important role in the harmonic analysis of $G/H$. They are also of much number theoretic interest since the local component, at a given place $v$, of an irreducible cuspidal representation of an adelic group, say $G(\mathbb A_k)$, with a non-vanishing \textit{H-period integral} is always distinguished for the pair $(G(k_v),H(k_v))$. 

Of specific interest is the case of a symmetric pair $(G,H)$, where $H$ is the subgroup of fixed points of an involution on $G$, and for such a symmetric space, the philosophy of Jacquet relates distinguished representations to Langlands functoriality. The most investigated and perhaps the best understood symmetric pair is $(GL_n(E),GL_n(F))$, where $E/F$ is a quadratic extension of non-archimedean local fields of characteristic zero. From the early works of Flicker, one knows that the symmetric space $GL_n(E)/GL_n(F)$ is a \textit{Gelfand pair}, i.e., it has the \textit{multiplicity one} property, and the question of distinction for this symmetric space is related to the study of the local Asai $L$-function, denoted by $L(s,\pi,As)$, as well as that it is connected to the \textit{base change lift} from the quasi-split unitary group in $n$ variables defined with respect to $E/F$ \cite{fli88,fli91,fli93}. Distinction for $(GL_n(E),GL_n(F))$ was also known to be intimately connected to other arithmetical invariants such as the \textit{epsilon factor for pairs} of Jacquet, Piatetski-Shapiro and Shalika \cite{jps83}, and the early works in this direction are due to Hakim and Ok \cite{hak91,ok97}. There has been a lot of progress along all these lines ever since and here we only refer to a sample of the concerned works \cite{akt04,kab04,ar05,ana08,mat09b,mat10,mat09a,off11,gur15,ho15,mo16}.

One particularly satisfying feature for distinction for the pair $(GL_n(E),GL_n(F))$ is that not only that 
\[\dim {\rm Hom}_{GL_n(F)}(\pi,1) \leq 1\]
for any irreducible admissible representation $\pi$ of $GL_n(E)$, but also that when the above dimension equals one, i.e., when $\pi$ is distinguished with respect to $GL_n(F)$, and when $\pi$ is a unitary representation which admits a Whittaker model, an explicit $GL_n(F)$-invariant linear form on $\pi$ can be realized on the Whittaker model of $\pi$. To this end, suppose 
$\psi:E \rightarrow \mathbb C^\times$ is a non-trivial additive character of $E$ that is trivial on $F$ and consider the $\psi$-Whittaker model $W(\pi,\psi)$ of $\pi$. Then the unique, up to multiplication by scalars, $GL_n(F)$-invariant linear form on $\pi$ can be written down on $W(\pi,\psi)$ as
\[\ell(W)=\int_{N_n(F)\backslash P_n(F)} W(p)dp= \int_{N_{n-1}(F)\backslash GL_{n-1}(F)} W(h)dh,\]
where $P_r$ denotes the mirabolic subgroup of $GL_r$ consisting of $r\times r$ invertible matrices whose last row equals $(0,\ldots,0,1)$, and $N_r$ is the maximal unipotent subgroup of $GL_r$ consisting of $r\times r$ unipotent upper triangular matrices  \cite{fli88,ok97,akt04}.  We remark here that though the integral defining the distinguishing linear form $\ell$ above is known to be convergent only under the assumption that the representation $\pi$ is unitary, even in the non-unitary case we could make sense of the above linear form via analytic continuation to $s=1$ of the linear form
\[\ell_s(W)=\int_{N_n(F)\backslash P_n(F)} W(p)|\det p|_F^{s-1}dp= \int_{N_{n-1}(F)\backslash GL_{n-1}(F)} W(h)|\det h|_F^{s-1}dh.\] 
For further details, we refer to Section  \ref{sec-nonunitary}. This is the \textit{local period} in the title of this paper. We summarize a few of the key properties of distinction for the pair $(GL_n(E),GL_n(F))$ in Section  \ref{sec-dist}.

There is an equally satisfying feature for $\psi$-distinction for the pair $(GL_n(E),N_n(E))$, where a non-trivial additive character $\psi:E \rightarrow \mathbb C^\times$, of conductor zero, is seen as a character of $N_n(E)$ by defining
\[\psi(n) = \psi(\sum_{i=1}^{n-1} n_{i,i+1}),\]
where $n=(n_{i,j}) \in N_n(E)$. Like before, not only do we have \cite{gk75,js83}
\[\dim {\rm Hom}_{N_n(E)}(\pi,\psi) \leq 1\]
for any irreducible admissible representation $\pi$ of $GL_n(E)$, but also when the above dimension equals one, i.e., when $\pi$ admits a $\psi$-Whittaker model, the unique, up to multiplication by scalars, $(N_n(E),\psi)$-invariant linear form, say $\Lambda$, on $\pi$ can be further explicated in the sense that an explicit vector on which $\Lambda$ is non-zero, satisfying several nice properties, can be realized in the Whittaker model of $\pi$. This is the \textit{essential vector} of an irreducible admissible generic representation $\pi$ of $GL_n(E)$ \cite{jps81}. We denote it by $W_\pi^0$ in the sequel. We extract a few important properties of the essential vector in \S \ref{subsec-essential}, supplementing with more details wherever necessary. 

It may be relevant at this point to mention the recent work of Lapid and Mao on \textit{model transitions}. In \cite{lm15b}, they initiate a study of model transitions in a very general context. In our context, when we consider the problem of distinction for the triples $(GL_n(E),GL_n(F),1)$ and $(GL_n(E),N_n(E),\psi)$, an irreducible admissible generic representation $\pi$ of $GL_n(E)$ which is distinguished with respect to $GL_n(F)$ affords two models; the Whittaker model given by
\[W_v(g) = \Lambda(\pi(g)v),\] and
\[f_v(g) \mapsto \ell(\pi(g)v),\]
for $v \in \pi$. Up to multiplication by scalars, there is a unique isomorphism between these two models, and by definition a model transition is a $GL_n(E)$-equivariant integral transform from one model to the other. The import of explicitly realizing the $GL_n(F)$-invariant functional $\ell$ on $W(\pi,\psi)$ is that 
\[W \mapsto \int_{N_n(F)\backslash P_n(F)} W(pg)dp\] 
is a model transition. We mention in passing that for a ``relatively cuspidal" (with respect to $GL_n(F)$) representation $\pi$ of $GL_n(E)$, an inverse model transition is given by \cite[Corollary 5.1]{off11}
\[f \mapsto \int_{N_n(F) \backslash N_n(E)} f(ng)\psi^{-1}(n)dn.\]

The main goal of the present paper is to highlight one more aspect of the interplay between the two models afforded by an irreducible admissible generic representation $\pi$ of $GL_n(E)$ which is distinguished with respect to $GL_n(F)$. Our first theorem asserts that the essential vector in the $\psi$-Whittaker model of $\pi$ is a test vector for the local period $\ell$. In fact, the value of the local period at the essential vector is an unramified Asai $L$-value. 

In the following theorem, $\pi_u$ denotes the unramified standard module attached in \cite{mat13} to an irreducible admissible generic representation $\pi$ of $GL_n(E)$ that  is ramified (cf. Theorem \ref{essentialformula}). 

\begin{theorem}\label{1}
Let $\pi$ be an irreducible admissible generic representation of $GL_n(E)$ which is distinguished with respect to $GL_n(F)$. Let $\psi: E\rightarrow \mathbb C^\times$ be an additive character of $E$ that is trivial on $F$ of conductor zero. Let $W_\pi^0 \in W(\pi,\psi)$ be the essential vector of $\pi$. Then, 
\[\ell(W_\pi^0)=\int_{N_n(F)\backslash P_n(F)} W_\pi^0(p)dp \neq 0.\] 
Furthermore,
\[\ell(W_\pi^0) =
\begin{cases}
L(1,\pi_u,As) &\text{if $\pi$ is ramified,} \\
\frac{L(1,\pi,As)}{L(n,1_{F^\times})} &\text{if $\pi$ is unramified}. 
\end{cases}
\]
\end{theorem}

\begin{remark}
The statement of Theorem \ref{1} depends on the choice of Haar measures. Throughout this paper, on all the groups considered, the chosen Haar measure is such that their respective maximal compact subgroups have volume one. More precise details are given in Section  \ref{prelim} and at the start of the proof of Theorem \ref{ramified-computation}.
\end{remark}

\begin{remark}
In Theorem \ref{1}, if $\pi$ is assumed to be unitary as well, we do not need to assume that $\pi$ is distinguished with respect to $GL_n(F)$ to compute $\ell(W_\pi^0)$ (cf. Theorem \ref{ramified-computation} and Theorem \ref{unramified-computation}). However, in the non-unitary generic case we do need this assumption (cf. Section  \ref{sec-nonunitary}).
\end{remark}

Another linear form on the representation $\pi$ that is closely related to $\ell$, once again given on the $\psi$-Whittaker model, is
\[\ell^\prime(W)=\int_{N_n(F)\backslash P_n(F)} W(w~ {^t}p^{-1})dp,\]
where $w$ is the longest Weyl element of $GL_n(E)$ with $1$'s on the anti-diagonal and $0$'s elsewhere. It is known that if $\pi$ is an irreducible admissible generic representation of $GL_n(E)$ which is distinguished with respect to $GL_n(F)$ then there exists a non-zero constant $c(\pi) \in \mathbb C\backslash \{0\}$, independent of $\psi$, such that $\ell^\prime=c(\pi)\ell$ (cf. Proposition \ref{constant-pi}). Our next result computes $c(\pi)$ to be $1$, as an application of Theorem \ref{1}.

\begin{theorem}\label{2}
Let $\pi$ be an irreducible admissible generic representation of $GL_n(E)$ which is distinguished with respect to $GL_n(F)$. Let $\psi: E\rightarrow \mathbb C^\times$ be any additive character of $E$ that is trivial on $F$. Then, for $W \in W(\pi,\psi)$,
\[\int_{N_n(F)\backslash P_n(F)} W(p)dp = \int_{N_n(F)\backslash P_n(F)} W(w~ {^t}p^{-1})dp.\]
\end{theorem}

\begin{remark}
If $\pi$ is an irreducible admissible generic representation of $GL_n(E)$ such that $\ell^\prime=\ell$ then it follows that the linear form $\ell$ is in fact $GL_n(F)$-invariant and thus $\pi$ is distinguished with respect to $GL_n(F)$. Indeed, $P_n(F)$ together with ${^t}P_n(F)$ generate $GL_n(F)$ \cite[theorem 1.3]{akt04}.
\end{remark}
 
Theorem \ref{2} is stated as \cite[Corollary 7.2]{off11} in the work of Offen \cite{off11}, but the proof there is valid only for relatively supercuspidal representations (see Remark \ref{rmk1} for further explanations). However, such a result is useful and in fact has already been used in a recent work of Lapid and Mao (\cite[Corollary 7.2]{off11} is referred to as \cite[Lemma 1.1]{lm14} in their paper, see Remark \ref{rmk2}). It seems most natural to prove such a theorem by evaluating the forms on a suitable test vector as we do in this paper. We may also remark at this point that Theorem \ref{2} is equivalent to saying that the $GL_n(F)$-invariant linear form $\ell$, on the $GL_n(E)$-representation $\pi$, is invariant under a larger group $GL(n,F) \rtimes \mathbb Z/2$, when $\pi$ is canonically considered as a representation of $GL(n,E) \rtimes \mathbb Z/2$ (see Remark \ref{rmk-dp} for more details). 

The key ingredient in the proof of Theorem \ref{1} is an explicit understanding of the essential vector of $\pi$ and in particular the formula to compute its value on the mirabolic subgroup \cite{mat13,miy14}. To conclude Theorem \ref{2}, in addition to a fine knowledge of the properties of the essential vector, we also need to appeal to the main result of \cite{mo16}, which relates distinction for the pair $(GL_n(E),GL_n(F))$ to the behavior of epsilon factor for pairs. Both Theorem \ref{1} and Theorem \ref{2} are proved in Section  \ref{sec-test}, under the assumption that $\pi$ is a unitary representation. The general case is taken up in Section  \ref{sec-nonunitary}. 

\section{Preliminaries}\label{prelim}

Let $F$ be a non-archimedean local field of characteristic zero. Let $\o_F$ be its ring of integers and $\p_F$ the unique maximal ideal of $\o_F$. Let $\w_F$ be a uniformizing parameter of $F$; thus $\p_F=\w_F\o_F$. We denote the cardinality of $\o_F/\p_F$ by $q_F$. Let $|\cdot|_F$ denote the normalized absolute value of $F$; thus $|\w_F|_F=q_F^{-1}$. Let  $E$ be a quadratic extension of $F$. Similarly we have its associated objects $\o_E$, $\p_E$, $\w_E$, $q_E$, and $|\cdot|_E$. We denote by $\sigma$ the non-trivial element of the Galois group ${\rm Gal}(E/F)$.

We set $G_n=GL(n,.)$. Let $\nu_E$ (resp. $\nu_F$) denote the character $|\det(\ \cdot\ )|_E$ (resp. $|\det(\ \cdot\ )|_F$) on $G_n(E)$ (resp. $G_n(F)$). We fix the Haar measure $dg$ on $G_n(E)$ such 
that $dg(G_n(\o_E))=1$, and the Haar measure $dh$ on $G_n(F)$ such that $dh(G_n(\o_F))=1$. Let $B_n(E)$ be the Borel subgroup of upper triangular matrices in $G_n(E)$, and $N_n(E)$ its unipotent radical. If $\psi$ is a non-trivial character of $E$, we continue to denote by $\psi$  the character of $N_n(E)$ given by $n\mapsto \psi(\sum_{i=1}^{n-1} n_{i,i+1})$. The maximal torus of $G_n(E)$ consisting of all diagonal matrices is denoted by $A_n(E)$. 

We denote by $\times$ the normalized parabolic induction. We will consider representations parabolically induced from essentially square integrable representations of Levi subgroups of parabolic subgroups of $G_n(E)$. First we recall that an irreducible admissible 
representation $\d$ of $G_n(E)$ is called \textit{essentially square integrable} if, after twisting by a character if necessary, it is \textit{square integrable}; i.e., it can be realized as a subrepresentation of $L^2(Z_n(E)\backslash G_n(E))$, where $Z_n(E)$ is the center of $G_n(E)$. According to \cite[Theorem 9.3]{zel80}, such a representation is the unique irreducible quotient of representations of the form 
\[\nu_E^{1-k}\rho \times \dots \times \rho,\] where $\rho$ is a supercuspidal representation of $G_r(E)$, with $n=rk$. We denote by 
$[\nu_E^{1-k}\rho , \dots ,\rho]$ such a quotient. We say that 
\[\d= [\nu_E^{1-k}\rho ,\dots ,\rho] {\rm ~\textit{precedes}~} \d^\prime= [\nu_E^{1-k^\prime}\rho^\prime , \dots, \rho^\prime] \]
if $\rho^\prime=\nu_E^l\rho$ for $\mathrm{max}(1,k^\prime-k+1)\leq l \leq k^\prime$, which we will denote by $\d\prec \d^\prime$. We say that $\d$ and $\d^\prime$ are \textit{linked} if either 
$\d\prec \d^\prime$ or $\d^\prime \prec \d$, and \textit{unlinked} otherwise.

An irreducible admissible representation $\pi$ of $G_n(E)$ is called \textit{generic} if 
\[{\rm Hom}_{N_n(E)}(\pi,\psi)\neq 0,\] in which case the above space is known to be one dimensional by a well-known result due to Gelfand and Kazhdan (cf. \cite{gk75}). If $\pi$ is an irreducible admissible generic representation of $G_n(E)$, by a result of Zelevinsky, we can write $\pi$ uniquely as a commutative product $$\pi=\d_1\times \dots \times \d_t$$ of unlinked essentially square integrable representations (cf. \cite[Theorem 9.7]{zel80}).
 
A more general class of representations is that of \textit{standard modules}, and these are the representations of the form 
$\d_1\times \dots \times \d_t$, with $\d_i$ essentially square integrable satisfying $\d_i\not\prec \d_j$ if $i<j$. For a standard module, 
the representations $\d_i$ are unique up to a reordering preserving the ``non-preceding'' condition. Standard modules $\pi$ also satisfy  \[\dim {\rm Hom}_{N_n(E)}(\pi,\psi)=1,\] and the intertwining operator 
from $\pi$ to ${\rm Ind}_{N_n(E)}^{G_n(E)}(\psi)$ is still injective \cite{js83}. The \textit{Whittaker model} of $\pi$, by definition, is the image of this embedding. We denote the Whittaker model of $\pi$ with respect to $\psi$ by $W(\pi,\psi)$.

For a smooth representation $\pi$ of $G_n(E)$, we denote by $\widetilde{\pi}$, the representation 
\[g\mapsto \pi(^t\!g^{-1})\]
of $G_n(E)$. If $\pi$ is moreover irreducible, then by \cite{gk75}, the representation $\widetilde{\pi}$ is isomorphic 
to the \textit{contragredient} of $\pi$. If $\pi$ is a standard module of $G_n(E)$, then for $W \in W(\pi,\psi)$, we define
\[\widetilde{W}(g)=W(w ~{^t}g^{-1}),\]
where $w$ is the longest Weyl element in $G_n(E)$ with $1$'s on the anti-diagonal and $0$'s elsewhere. It is easy to see then that $\widetilde{W} \in W(\widetilde{\pi},\psi^{-1})$.

If an admissible representation $\pi$ of $G_n(E)$ has a central character (for example a representation parabolically 
induced from an irreducible representation of a Levi), we denote it by $\omega_\pi$. If $\iota: G_n(E) \rightarrow G_n(E)$ is a homomorphism, by $\pi^\iota$ we mean the representation given by $\pi^\iota(g)=\pi(\iota(g))$. In the sequel, we will typically apply this to the homomorphism given by the Galois involution $\sigma$.

\section{Asai and Rankin-Selberg integrals}

Let $\pi$ and $\pi^\prime$ be two standard modules of $G_n(E)$ and $G_m(E)$ respectively. We denote by $L(s,\pi \times \pi^\prime)$ the $L$-factor attached to such a pair in \cite{jps83}. 
As we shall need more precise information when $m=n-1$, we recall the construction in this case. 

It is shown in \cite{jps83} that there is a real number $r_{\pi,\pi^\prime}$ such that $\mathrm{Re}(s)\geq r_{\pi,\pi^\prime}$ implies that for $W\in W(\pi,\psi)$ and 
$W'\in W(\pi^\prime,\psi^{-1})$, the integrals 
\[I(s,W,W')=\int_{N_{n-1}(E)\backslash G_{n-1}(E)} W\left(\begin{array}{cc} g & 0 \\ 0 & 1 \end{array}\right)W'(g)\nu_E(g)^{s-\frac{1}{2}} dg\] 
are absolutely convergent. In fact they extend to elements of $\C(q_E^{-s})$, i.e., rational functions in $q_E^{-s}$, and the vector space they span, as $W$ and $W'$ vary, is a fractional ideal $I(\pi,\pi^\prime)$ 
of the ring of Laurent polynomials $\C[q_E^{\pm s}]$. This fractional ideal is generated by a single element of the form $P(q_E^{-s})^{-1}$, for a polynomial $P$, and if $P$ is normalized so that $P(0)=1$, we denote this generator by $L(s,\pi \times \pi^\prime)$.

The following lemma is \cite[Lemma 2.3]{mo16} taking $u=1$; it is a consequence of the standard results of \cite{jps83}.

\begin{lemma}\label{poleA}
If $\d_1$ and $\d_2$ are two essentially square integrable representations of $G_n(E)$, then $L(s,\d_1 \times \d_2)$ has a pole at $s=1$ if and only if $\d_1\prec \widetilde{\d_2}$.
\end{lemma}

We will also need to consider the Asai integrals defined by Flicker in \cite{fli93}. We refer to \cite{mat09b} for a detailed description of the basic properties of these integrals. We start with a non-trivial additive character $\psi$ of $E$ which is trivial on $F$. Let  
\[\e_n=(0,\dots,0,1) \in F^n.\] 
If $\pi$ is a standard module of $G_n(E)$, 
there is a real number $r_\pi$ such that for $\mathrm{Re}(s)\geq r_{\pi}$, $W\in W(\pi,\psi)$ and $\Phi\in \mathcal{C}_c^\infty(F^n)$, the integrals 
\[I(s,W,\Phi)=\int_{N_n(F)\backslash G_n(F)} W(h)\Phi(\e_n h)\nu_F(h)^s dh\]
are absolutely convergent. In fact, exactly as in the case of Rankin-Selberg integrals above, they extend to elements of $\C(q_F^{-s})$, and the vector space they span, as $W$ and $\Phi$ vary, is a fractional ideal $I(\pi)$ of $\C[q_F^{\pm s}]$ with a unique (normalized) generator which we denote by $L(s,\pi,As)$.

The following lemma can be found in Section 3 of \cite{mat09b}.
\begin{lemma}\label{inclusion}
The integrals 
$$I_{(0)}(s,W)=\int_{N_n(F)\backslash P_n(F)} W(h)\nu_F(h)^{s-1} dh=\int_{N_{n-1}(F)\backslash G_{n-1}(F)} W(h)\nu_F(h)^{s-1} dh$$
are in fact of the form $I(s,W,\Phi)$, for some well chosen $\Phi$, and hence they also belong to $I(\pi)$.
\end{lemma}

The multiplicativity relation of the Asai $L$-functions of Flicker
is a consequence of \cite[Theorem 4.26]{mat09b} and \cite[Theorem 5.2]{mat09a}.

\begin{theorem}\label{multiplicativity-standard}
If $\pi=\d_1\times \dots \times \d_t$ is a standard module of $G_n(E)$, then 
\[L(s,\pi,As)=\prod_i L(s,\d_i,As)\prod_{j<k} L(s,\d_j \times \d_k^\sigma).\]
\end{theorem}

\begin{remark}
We recall that Theorem \ref{multiplicativity-standard} implies that 
the Asai $L$-function of $\pi$ defined as above is equal to that of the Langlands parameter of the Langlands quotient of $\pi$ \cite[Theorem 5.3]{mat09a}. Therefore, we will not be very precise about the definition of the Asai $L$-function that we refer to, as they both agree. In particular, for a generic unramified representation $\pi$ of $G_n(E)$, in what follows we appeal to a result due to Flicker \cite[Proposition 3]{fli88} to assert that
\[I(s,W_\pi^0,\Phi_0)=L(s,\pi,As),\]
where $W_\pi^0$ is the normalized spherical vector of $\pi$ and $\Phi_0$ is the characteristic function of $\o_F^n$, though Flicker's result originally had the Asai $L$-function of the Langlands parameter.
\end{remark}

We will need the following consequence of Lemma \ref{poleA} and  \cite[Theorem 6]{kab04}.

\begin{lemma}\label{poleB}
Let $\d$ be an essentially square integrable representation of $G_n(E)$.  If $L(s,\d,As)$ has a pole at $s=1$ then $\d\prec \widetilde{\d}^\sigma$.
\end{lemma}

\begin{remark}
According to \cite[Theorem 6]{kab04}, we have the identity  
\[L(s,\delta^\prime \times \delta^{\prime\sigma})=L(s,\delta^\prime,As)L(s,\delta^\prime \otimes \kappa,As),\]
where $\kappa$ is any character of $E^\times$ which extends the non-trivial character of $F^\times/N_{E/F}(F^\times)$, for a square integrable representation $\delta^\prime$ of $G_n(E)$. Observe that the identity is valid for an essentially square integrable representation of $G_n(E)$ as well. Indeed, since an essentially square integrable representation $\delta$ of $G_n(E)$ is of the form $\delta \otimes \nu_E^a$, for $a \geq 0$, for a square integrable representation $\delta$ of $G_n(E)$, and we then have 
\[L(s,\delta \times \delta^\sigma)=L(s+2a,\delta^\prime \times \delta^{\prime\sigma})\]
and
\[L(s,\delta,As)=L(s+2a,\delta^\prime,As).\]
Note that the latter identity follows since $\nu_E=\nu_F^2$. 
\end{remark}

\section{Conductor and the essential vector}\label{sec-essential}

\subsection{A few lemmas on the conductor}

If $\psi$ is a non-trivial character of $E$, we denote by $n(\psi)$ its \textit{conductor}, i.e. the largest integer $k$ such that 
$\psi$ is trivial on $\p_E^{-k}$. If $\chi$ is a character of $E^*$, we denote by $f(\chi)$ its \textit{conductor}, i.e. $f(\chi)=0$ if 
$\chi$ is unramified, otherwise $f(\chi)$ is the smallest $m$ such that $\chi$ is trivial on $1+\p_E^m$.

\begin{lemma}\label{lem-even}
Suppose that $E$ is ramified over $F$, then any non-trivial character of $E$ that is trivial on $F$ is of even conductor.
\end{lemma}

\begin{proof}
Let $\psi$ be a character of $E/F$. The assumption that $E$ is ramified over $F$ implies that we can define an unramified character $\chi$ of $E^\times$ that is trivial on $F^\times$, simply 
by setting $\chi(\w_E)=-1$. Then, by \cite[Theorem 3]{fq73}, one has 
\[\epsilon(1/2,\chi,\psi)=1,\] where $\e$ is the 
local constant of Tate. By \cite[\S 3]{tat77}, one also has \[\epsilon(1/2,\chi,\psi)=\chi(\w_E)^{n(\psi)}\e(1/2,1_{E^*},\psi)=(-1)^{n(\psi)},\] and the result follows.
\end{proof}

\begin{remark}\label{rem-even}
Note that Lemma \ref{lem-even} can be seen directly too from the very definition of the conductor of an additive character. To this end, observe that if $\psi^\prime$ is an additive character of $F$, then,
\[
n(\psi^\prime \circ {\rm Tr}_{E/F}) =
\begin{cases}
n(\psi^\prime) &\text{if $E/F$ is unramified,} \\
2n(\psi^\prime)+f(E/F) &\text{if $E/F$ is ramified},
\end{cases}
\]
where $f(E/F)$ is the conductor of the quadratic extension $E/F$.  Any character of $E$ that is trivial on $F$ will be of the form
\[\psi(x)=\psi^\prime({\rm Tr}_{E/F}(\Delta x)),\]
for some $\psi^\prime : F \rightarrow \mathbb C^\times$, 
where $\Delta \in E^\times$ is an element of trace zero. Now we only need to observe that if $f(E/F)$ is odd then there is a uniformizer of trace zero and if $f(E/F)$ is even then there is a unit of trace zero. 
\end{remark}

As a consequence, we obtain the following lemma.

\begin{lemma}\label{level0}
There always exists a non-trivial character $\psi$ of $E$ that is trivial on $F$ such that $n(\psi)=0$.
\end{lemma}

\begin{proof}
Take $\psi^\prime$ a character of $E$ trivial on $F$ of conductor, say $k$. 
If $E$ is unramified over $F$, then we define $\psi(x)=\psi^\prime(\w_F^{-k} x )$. If $E$ is ramified over $F$, then by the lemma above, $k=2l$ is even, and in this case we set $\psi(x)=\psi^\prime(\w_F^{-l}x)$.
\end{proof}

We now recall that if $\pi$ is an irreducible admissible representation of $G_n(E)$, and $\psi$ is a character of conductor zero, then its Godement-Jacquet epsilon factor $\e(s,\pi,\psi)$ (see \cite{gj72}) is of the form $\alpha q^{-f(\pi)s}$, where $f(\pi)$ is an integer independent 
of $\psi$ (since we chose $n(\psi)=0$). We call this integer $f(\pi)$ the \textit{conductor} of $\pi$.

We will call an irreducible admissible representation $\pi$ of $G_n(E)$ \textit{distinguished} with respect to $G_n(F)$ if the space 
\[{\rm Hom}_{G_n(F)}(\pi,\C)\] 
is not zero. In the sequel, we will need to know the parity of the conductor of a distinguished representation when $E/F$ is ramified. We take this up in the next lemma.

\begin{lemma}\label{evenconductor}
Let $E/F$ be ramified. If $\pi$ is an irreducible admissible representation of $G_n(E)$ distinguished with respect to $G_n(F)$, then its conductor is even.
\end{lemma}

\begin{proof}
To prove the lemma we make use of \cite[Theorem 3.6]{mo16}, from which it follows that 
\[\epsilon(1/2,\pi,\psi)=1\] if $\pi$ is distinguished with respect to $G_n(F)$ and $\psi$ is a non-trivial character of $E$ that is trivial when restricted to $F$. By Lemma \ref{level0}, we may choose such a $\psi$ with $n(\psi)=0$. As in the proof of that lemma let $\mu$ be an unramified quadratic character of $E^\times$ which is trivial on $F^\times$. Clearly $\pi \otimes \mu$ is also distinguished. Therefore,
\[\epsilon(1/2,\pi\otimes \mu,\psi)=1=\epsilon(1/2,\pi,\psi).\]
On the other hand, we have 
\[\epsilon(1/2,\pi\otimes \mu,\psi)=\mu(\varpi_E)^{f(\pi)}\epsilon(1/2,\pi,\psi),\]
by the well-known behavior of epsilon factors under twisting by an unramified character \cite[\S 3]{tat77}.
It follows that $f(\pi)$ is even. 
\end{proof}

\begin{remark}\label{rem-mo}
According to \cite[Theorem 3.6]{mo16}, if $\pi_1$ (resp. $\pi_2$) is an irreducible admissible representation of $G_{n_1}(E)$ (resp. $G_{n_2}(E)$) which is distinguished with respect to $G_{n_1}(F)$ (resp. $G_{n_2}(F)$), then
\[\epsilon(1/2,\pi_1 \times \pi_2,\psi)=1\]
where $\psi$ is a character of $E$ which is trivial on $F$. This result was established for cuspidal representations 
in Youngbin Ok's PhD thesis \cite{ok97}, where a cuspidal relative converse theorem for the pair $(G_n(E),G_n(F))$ was also proved. Both these results from \cite{ok97} are generalized in \cite{mo16}, to which we refer for more explanations about this topic. In fact, the aforementioned 
result of \cite{mo16} is \cite[Conjecture 5.1]{ana08}.  
\end{remark}

\begin{remark}
In the proof of Lemma \ref{evenconductor}, we made use of the effect of distinction on the epsilon factor and further knowledge of epsilon factors under twisting by unramified characters, since this seems to be the easiest way to conclude the lemma. Thus, our proof is in the spirit of the proof of Lemma \ref{lem-even}. It will be interesting to give a proof of Lemma \ref{evenconductor} by making use only of the definition of a distinguished representation. Note that in the case of Lemma \ref{lem-even}, Remark \ref{rem-even} provides a direct proof.
\end{remark}

\subsection{The essential vector of a generic representation}\label{subsec-essential}

In this section we fix a character $\psi$ of $E$ with $n(\psi)=0$. Let $\pi$ be an irreducible generic representation of $G_n(E)$. Denote by $H_n(k)$ the group $G_n(\o_E)$ if $k=0$, whereas if $k>0$ let $H_n(k)$ be the subgroup of $G_n(\o_E)$ given by 
\[H_n(k)=\left\{\begin{pmatrix} g & x \\ l & t \end{pmatrix} \in G_n(\o_E) \mid g \in G_{n-1}(\o_E), l \in \p_E^k, t\in 1+\p_E^k \right\}.\]
One of the main results of \cite{jps81} 
is that $f(\pi)$ is the smallest non-negative integer $k$ such that $\pi$ contains a vector fixed by $H_n(k)$, and moreover in this case, there is only
one Whittaker function in the Whittaker model $W(\pi,\psi)$ fixed by $H_n(f(\pi))$, and equal to $1$ on this subgroup \cite[Theorem 5.1]{jps81}.

For an irreducible admissible generic representation $\pi$ of $G_n(E)$, we denote by $W_\pi^0$ the unique right $H_n(f(\pi))$-invariant Whittaker 
function in $W(\pi,\psi)$ which equals $1$ on $H_n(f(\pi))$. This is called the \textit{essential vector} of $\pi$. More generally, if $\pi$ is an unramified standard module, reducible or not, we denote by $W_\pi^0$ the normalized spherical vector in its Whittaker model (which thus agrees with the essential vector when $\pi$ is irreducible). 

Let us now recall an important property of the essential vector of an irreducible  generic representation $\pi$ of $G_n(E)$. Whenever $\pi^\prime$ is an unramified standard module of 
$G_{n-1}(E)$, one has \cite[Theorem 4.1]{jps81}:
\[I(s,W_{\pi}^0,W_{\pi^\prime}^0)=L(s,\pi \times \pi^\prime),\]
and moreover the essential vector is the unique Whittaker function in $W(\pi,\psi)$ satisfying the above relation for all unramified standard modules $\pi^\prime$ of $G_{n-1}(E)$. 

We will need to use the following result which relates the essential vector of an irreducible representation to that of its contragredient. This result is extracted from of \cite[\S 5.3]{jps81} and we add the necessary details here.

\begin{proposition}\label{dual-of-essential}
Let $\pi$ be an irreducible generic representation with conductor $m$. Set 
\[p_m=\left(\begin{array}{cc} \w_E^m I_{n-1} & 0 \\ 0 & 1 \end{array}\right)\in G_n(E).\] Then we have the identity
\[\widetilde{W_\pi^0}=\e(1/2,\pi,\psi)^{n-1} \widetilde{\pi}(p_m)W_{\widetilde{\pi}}^0.\]
\end{proposition}

\begin{proof}
Let $\nu_E^{t_1}\otimes \dots \otimes \nu_E^{t_{n-1}}$ be an unramified character of $A_{n-1}(E)$ such that 
\[\pi_t=\nu_E^{t_1}\times \dots \times \nu_E^{t_{n-1}}\] is an unramified standard module of $G_{n-1}(E)$. By \cite[\S 5.3]{jps81}, one has: 
\[I(s,\widetilde{W_\pi^0},W_{\pi_t}^0)=C^{n-1}q_{E}^{-m(n-1)}q_{E}^{m(n-1)s}q_E^{m \sum_{i=1}^{n-1} t_i}L(s,\widetilde{\pi} \times \pi_t),\]
where $C=\epsilon(0,\pi,\psi)$ (cf. \cite[Remark 5.4]{jps81}). 
On the other hand, by an easy change of variable, we see that: 
\begin{align*}
I(s,\widetilde{\pi}(p_m)W_{\widetilde{\pi}}^0,W_{\pi_t}^0) &=  
\nu_E(p_m)^{1/2-s}\omega_{\pi_t}^{-1}(\w_E^mI_{m-1})I(s,W_{\widetilde{\pi}}^0,W_{\pi_t}^0) \\
&= q_{E}^{-m(n-1)/2}q_{E}^{m(n-1)s}q_E^{m\sum_{i=1}^{n-1} t_i}I(s,W_{\widetilde{\pi}}^0,W_{\pi_t}^0) \\
&= q_{E}^{-m(n-1)/2}q_{E}^{m(n-1)s}q_E^{m\sum_{i=1}^{n-1} t_i}L(s,\widetilde{\pi} \times \pi_t).
\end{align*}
It follows that  
\begin{align*}
I(s,\widetilde{W_\pi^0},W_{\pi_t}^0) &= C^{n-1}q_{E}^{-m(n-1)/2}I(s,\widetilde{\pi}(p_m)W_{\widetilde{\pi}}^0,W_{\pi_t}^0)\\
&= \e(1/2,\pi,\psi)^{n-1}I(s,\widetilde{\pi}(p_m)W_{\widetilde{\pi}}^0,W_{\pi_t}^0)\\
&= I(s,\e(1/2,\pi,\psi)^{n-1}\widetilde{\pi}(p_m)W_{\widetilde{\pi}}^0,W_{\pi_t}^0).
\end{align*}
By the uniqueness property of the essential vector, we deduce the identity 
\[\widetilde{W_\pi^0}=\e(1/2,\pi,\psi)^{n-1} \widetilde{\pi}(p_m)W_{\widetilde{\pi}}^0.\]
\end{proof}

We will make use of further information about the essential Whittaker function which can be found in 
\cite{mat13,miy14}. More precisely, in \cite{mat13}, if $\pi$ is ramified (i.e. $f(\pi)>0$), an unramified standard module $\pi_u$ of $G_r(E)$ with $r< n$, is associated to $\pi$ as follows. 

To say that 
\[\pi=[\nu_E^{1-a_1}\rho_1(\pi),\dots,\rho_1(\pi)]\times \dots \times [\nu_E^{1-a_t}\rho_t(\pi),\dots,\rho_t(\pi)]\]
is ramified is to say that there is at least one $i$ such that either $a_i> 1$ or $\rho_i(\pi)$ is not an unramified character of $GL(1,E)$. Now assume that $\pi$ is a ramified generic representation of $G_n(E)$ of the above form. Denote by $U(\pi)$ the subset of $\{\rho_1(\pi),\dots,\rho_t(\pi)\}$ consisting of the $\rho_i$'s which are unramified characters of $GL(1,E)$, and denote by $\chi_1,\dots,\chi_r$ its (maybe equal) elements ordered such that 
$\mathrm{Re}(\chi_{i})\geq \mathrm{Re}(\chi_{i+1})$. By definition, we set 
$$\pi_u=\chi_1\times \dots \times \chi_r.$$

The main result of \cite{mat13} is the following formula for the restriction of the essential vector $W_{\pi}^0$ to the diagonal torus $A_{n-1}(E)$ of $G_{n-1}(E)$. 

\begin{theorem}\label{essentialformula} 
Let $\pi$ be an irreducible ramified generic representation of $G_n(E)$. Let 
\[a=diag(a_1,\dots, a_{n-1})\in A_{n-1}(E)\] and 
\[a^\prime=diag(a_1,\dots, a_r)\in A_r(E).\] 
Then we have
\[W_\pi^0(diag(a,1))=W_{\pi_u}^0(a^\prime)\nu_E(a^\prime)^{(n-r)/2}1_{\o_E}(a_r)\prod_{r<i<n}1_{\o_E^*}(a_i).\]
\end{theorem}

\begin{remark}\label{rem-essential}
The essential vector of Jacquet, Piatetski-Shapiro and Shalika plays a fundamental role in number theory via the theory of \textit{newforms} for $GL_n$. Thus, an explicit understanding of its properties, a few of which are covered in this section, may be of independent interest, and can have potential applications in number theory. In this context, we refer to \cite[\S 7]{ven06},  where the $L^2$-norm of the essential vector is computed (in a very elegant way by appealing to invariant theory!), with applications to analytic number theory. 
\end{remark}

\section{Distinguished representations}\label{sec-dist}

We recall that an irreducible admissible representation $\pi$ of $G_n(E)$ is distinguished with respect to $G_n(F)$ if the space 
\[{\rm Hom}_{G_n(F)}(\pi,1)\]
of $G_n(F)$-invariant linear forms on $\pi$ is not zero. 

First we have the following basic result due to Flicker \cite[Propositions 11 \& 12]{fli91}.

\begin{proposition}\label{basic}
If $\pi$ is an irreducible admissible representation of $G_n(E)$ which is $G_n(F)$-distinguished then \[\dim {\rm Hom}_{G_n(F)}(\pi,1) = 1,\] and
moreover $\widetilde{\pi}\simeq \pi^\sigma$.
\end{proposition}

As a consequence of the above result we deduce the following useful corollary.

\begin{corollary}\label{obvious}
Let $\pi$ be an irreducible admissible generic representation of $G_n(E)$ which is distinguished with respect to $G_n(F)$. Then we have $(\widetilde{\pi})_u \simeq \pi_u$.
\end{corollary}
\begin{proof}
It is clear from the definition of $\pi_u$ that $(\pi^\sigma)_u \simeq (\pi_u)^\sigma$. 
As any unramified character of $E^*$ is fixed by $\sigma$, the relation $\widetilde{\pi} \simeq \pi^\sigma$ implies that 
\[(\widetilde{\pi})_u \simeq (\pi^\sigma)_u \simeq (\pi_u)^\sigma \simeq \pi_u.\]
\end{proof}

We now recall a result due to Youngbin Ok \cite[Theorem 3.1.2]{ok97}. We refer to \cite[Proposition 2.1]{mat10} and \cite[Theorem 3.1]{off11} for published proofs of Ok's result. When $\pi$ is tempered, another approach to Ok's theorem is given also in \cite[Theorem 1.1]{akt04}.

\begin{proposition}\label{P-invariance}
Let $\pi$ be an irreducible admissible representation of $G_n(E)$ which is distinguished with respect to $G_n(F)$. Then, 
\[{\rm Hom}_{G_n(F)}(\pi,1)={\rm Hom}_{P_n(F)}(\pi,1).\]
\end{proposition}

Now let $\pi$ be an irreducible unitary generic representation of $G_n(E)$. Let $\psi$ be a non-trivial character of $E$ which is trivial on $F$. Let $\ell$ and $\ell^\prime$ be the linear forms on $W(\pi,\psi)$ defined by 
\[\ell:W\mapsto I_{(0)}(1,W)=\int_{N_{n-1}(F)\backslash G_{n-1}(F)} W(h)dh,\]
and 
\[\ell^\prime:W\mapsto I_{(0)}(1,\widetilde{W})=\int_{N_{n-1}(F)\backslash G_{n-1}(F)} \widetilde{W}(h)dh.\]
Both these integrals are convergent according to a result of Flicker \cite[\S 4]{fli88}), and thus the linear forms above are well defined. Moreover, as $W(\pi,\psi)$ and 
$W(\widetilde{\pi},\psi^{-1})$ contain respectively ${\rm ind}_{N_n}^{G_n}(\psi)$ and ${\rm ind}_{N_n}^{G_n}(\psi^{-1})$ according to \cite[\S 5.15 \& \S 5.16]{bz76}, the linear forms $\ell$ and $\ell^\prime$ are nonzero. Observe also that as $\ell$ is naturally $P_n(F)$-invariant whereas $\ell^\prime$ is $^t\! P_n(F)$-invariant, and therefore if $\pi$ is distinguished (hence $\widetilde{\pi}$ as well), then according to Proposition \ref{P-invariance}, they are $G_n(F)$-invariant. By the multiplicity one result for invariant linear functionals, together with some further easy computations, Offen shows that the above observation has the consequence that these two linear forms differ by a non-zero scalar which depends only on the representation $\pi$. We refer to Corollary 4.1 and Remark 3 just after it in \cite{off11} for further details (see also \cite[Theorem 1.3]{akt04}). 

\begin{proposition}\label{constant-pi}
Let $\pi$ be an irreducible admissible generic representation of $G_n(E)$ which is distinguished with respect to $G_n(F)$. Then there exists a non-zero constant $c(\pi)\in \C \backslash \{0\}$, independent of $\psi$, such that $\ell^\prime = c(\pi)\ell$.
\end{proposition}

\begin{remark}\label{rmk1}
According to \cite[Corollary 7.2]{off11}, the constant $c(\pi)$ in Proposition \ref{constant-pi} is equal to $1$. However the proof in \cite{off11} relies on the fact that the gamma factor at $s=1/2$ in Theorem 0.1 of [ibid] is $1$ for an irreducible admissible distinguished representation, but this fact is not true in general (see \cite{mo16} for further explanations). Hence the proof in Offen's work gives $c(\pi)=1$ for relatively cuspidal representations \cite[Corollary 6.1]{off11}. It would extend using the results of \cite{mo16}, as soon as one knows that relatively cuspidal representations are always tempered, which is expected to be true.
\end{remark}

\begin{remark}\label{rmk2}
As mentioned in the introduction, the fact that $c(\pi)$ is always equal to $1$, i.e., $\ell^\prime = \ell$, can be thought of as a new type of functional equation (cf. Theorem \ref{constant}). This and similar other functional equations are important and appear in the recent works of Lapid and Mao \cite{lm14,lm15,lm15b}, and also in \cite{mor16}. In particular, Theorem \ref{constant} is used and referred to as 
\cite[Lemma 1.1]{lm14} in their work. The main result of \cite{lm14} deals with a similar proportionality constant in a more general situation involving distinguished representations. See also \cite{lm15} for results and conjectures about a proportionality constant arising out of the multiplicity one result for Whittaker functionals. A consequence of their work on \textit{model transitions} is a functional equation which is in the spirit of Theorem \ref{constant} of this paper \cite[Proposition 3.9]{lm15b}.
\end{remark}

In the next section we will give a simple proof of the fact that the constant $c(\pi)$ is always $1$ as a consequence of our explicit computation of local periods on appropriate test vectors (cf. Theorem \ref{constant}). Its proof will also make use of \cite[Theorem 3.6]{mo16} which is about the epsilon factor of a $G_n(F)$-distinguished representation of $G_n(E)$ (cf. Remark \ref{rem-mo}). 

\section{Explicit test vectors and computation of local periods}\label{sec-test}

In this section we show that the essential vector of an irreducible unitary generic representation $\pi$ of $G_n(E)$ is a test vector for the linear form given on the Whittaker model of $\pi$ by
\[\ell(W)=\int_{N_{n-1}(F)\backslash G_{n-1}(F)} W\left(\begin{array}{cc}g & 0 \\ 0 & 1 \end{array} \right)dg.\]
In fact, we show that the value of the above linear form on the essential vector is related to a a certain $L$-value.

\subsection{Test vectors}

In this section, we prove Theorem \ref{1}, under the assumption that the irreducible admissible representation $\pi$ of $G_n(E)$ is unitary.

We start by appealing to Lemma \ref{level0} to choose a character $\psi$ of $E$ of conductor $0$ which is trivial on $F$. Now let $\pi$ be an irreducible unitary generic representation and let $W_\pi^0$ be the essential vector in $W(\pi,\psi)$. We compute $\ell(W_\pi^0)$ (whether 
$\pi$ is distinguished with respect to $G_n(F)$ or not). 

The computation when $\pi$ is ramified 
is slightly different from that when $\pi$ is unramified. We start with the ramified case.

\begin{theorem}\label{ramified-computation}
Let $\pi$ be an irreducible generic unitary representation of $G_n(E)$ which is ramified. Then we have \[\ell(W_\pi^0)=L(1,\pi_u, As) \neq 0.\]
\end{theorem}

\begin{proof}
In agreement with our normalization convention of Haar measures, we can choose the measure  $da=da_1\dots da_{n-1}$ on $A_{n-1}(F)$ such that $da_i(\o_F^\times)=1$, and that on $G_{n-1}(\o_F)$ such that $dk(G_{n-1}(\o_F))=1$. We have
\begin{align*}
\ell(W_\pi^0) &= \int_{N_{n-1}(F)\backslash G_{n-1}(F)} W_\pi^0 \left(\begin{array}{cc}g & 0 \\ 0 & 1 \end{array} \right) dg \\
&= \int_{A_{n-1}(F)}\int_{G_{n-1}(\o_F)} W_\pi^0 \left( \left(\begin{array}{cc} a & 0 \\ 0 & 1 \end{array} \right)k \right)\d_{B_{n-1}(F)}^{-1}(a)da ~dk \\
&= \int_{A_{n-1}(F)} W_\pi^0  \left(\begin{array}{cc} a & 0 \\ 0 & 1 \end{array} \right) \d_{B_{n-1}(F)}^{-1}(a)da .
\end{align*}
Now we proceed by invoking Theorem \ref{essentialformula}. Thus, 
\begin{align*} 
\ell(W_\pi^0) &= \int_{A_r(F)}W_{\pi_u}^0(a^\prime)\d_{B_{n-1}(F)}^{-1}  \left(\begin{array}{cc} a^\prime & 0 \\ 0 & I_{n-1-r} \end{array} \right) \nu_E(a^\prime)^{(n-r)/2}1_{\o_E}(a_r)da^\prime  \\
&= \int_{A_r(F)}W_{\pi_u}^0(a^\prime)\d_{B_r(F)}^{-1}(a^\prime)\nu_F(a^\prime)^{r-(n-1)}\nu_E(a^\prime)^{(n-r)/2}1_{\o_E}(a_r)da^\prime \\
&= \int_{A_r(F)}W_{\pi_u}^0(a^\prime)\d_{B_r(F)}^{-1}(a^\prime)\nu_F(a^\prime)1_{\o_F}(a_r)da^\prime \\
&= I(1,W_{\pi_u}^0,1_{\o_F^r}) \\
&= L(1,\pi_u, As),
\end{align*}
where the last equality follows thanks to \cite[Proposition 3]{fli88}. 

To be more precise the statement of the aforementioned proposition 
is for irreducible unramified generic representations, but its proof is valid for unramified standard modules as we now explain. 
The computation in the proof of \cite[Proposition 3]{fli88} uses Shintani's explicit formula for $W_{\pi_u}^0$, which is valid 
for any spherical Whittaker function \cite{shi76}. Moreover, if $\pi_u$ is a standard module, this will guarantee that the formula obtained by Flicker still gives the Asai $L$ function of $\pi_u$ (or equivalently of its Langlands quotient), by definition.
\end{proof}

We now do the computation when $\pi$ is unramified.

\begin{theorem}\label{unramified-computation}
Let $\pi$ be an irreducible generic unitary representation of $G_n(E)$ which is unramified. Then we have \[\ell(W_\pi^0)=\frac{L(1,\pi,As)}{L(n,1_{F^\times})}\neq 0.\]
\end{theorem}

\begin{proof}
We normalize the measure $da=da_1\dots da_n$ on $A_{n}(F)$ such that $da_i(\o_F^\times)=1$, and that on $G_n(\o_F)$ such that $dk(G_n(\o_F))=1$. Let $\Phi_0$ be the characteristic function of $\o_F^{n}$. 
By \cite[Proposition 3]{fli88}, we know that 
\[I(s,W_\pi^0,\Phi_0)=L(s,\pi,As).\]
On the other hand, we also have
\begin{align*}
I(s,W_\pi^0,\Phi_0) &= \int_{G_n(\o_F)}\int_{N_n\backslash P_n}W_\pi^0(pk)\nu_F(p)^{s-1}dp \int_{F^\times} \Phi_0((0,\dots,0,t)k)|t|_F^{ns}dt ~dk\\
&= \int_{N_n\backslash P_n}W_\pi^0(p)\nu_F(p)^{s-1}dp  \int_{F^\times} \int_{G_n(\o_F)} \Phi_0((0,\dots,0,t)k)|t|_F^{ns}dt ~dk\\
&= \int_{N_n\backslash P_n}W_\pi^0(p)\nu_F(p)^{s-1}dp  \int_{F^\times} \Phi_0(0,\dots,0,t)|t|_F^{ns}dt\\
&= I_{(0)}(s,W_\pi^0)L(ns,1_{F^\times}),
\end{align*}
where the last equality follows from Tate's thesis. Evaluating this equality at $s=1$ gives the result.
\end{proof}

\subsection{Computation of the constant $c(\pi)$}

The goal of this section is to prove Theorem \ref{2}, assuming further that the representation $\pi$ of $G_n(E)$ is unitary. As remarked before (cf. Remark \ref{rmk1}), when $\pi$ is an irreducible generic unitary representation that is relatively cuspidal, Theorem \ref{constant} is a consequence of Offen's work \cite{off11}.

Before we begin we introduce the following notation for the ease of exposition. Let
\[L^*(s,\pi,As)=
\begin{cases}
L(s,\pi_u,As) &\text{if $\pi$ is ramified,} \\
\frac{L(s,\pi,As)}{L(ns,1_{F^\times})} &\text{if $\pi$ is unramified}. 
\end{cases}
\] 
Here is the main theorem of this section.

\begin{theorem}\label{constant}
Let $\pi$ be an irreducible generic unitary representation of $G_n(E)$ which is distinguished with respect to $G_n(F)$. Then $c(\pi)=1$; i.e., $\ell^\prime = \ell$.
\end{theorem}

\begin{proof}
Since $c(\pi)$ does not depend on the choice of $\psi$, we can take $\psi$ to be trivial on $F$, and of conductor $0$. As mentioned earlier, such a choice is possible by Lemma \ref{level0}.

If $E/F$ is unramified, we can take $\w_E=\w_F$, and hence 
\[p_m=\left(\begin{array}{cc} \w_E^m I_{n-1} & 0 \\ 0 & 1 \end{array}\right)\in G_n(F).\] 
Therefore, by Proposition \ref{dual-of-essential}, and by observing that the epsilon factor is trivial since $\pi$ is distinguished with respect to $G_n(F)$ \cite[Theorem 3.6]{mo16}, we get
\begin{equation}\label{eq1}
\ell^\prime(W_\pi^0)=\ell(\widetilde{W_\pi^0})=\ell(\widetilde{\pi}(p_m)W_{\widetilde{\pi}}^0)= 
\ell(W_{\widetilde{\pi}}^0),
\end{equation}
where the last equality follows by a change of variable. 
But then we already know, by Theorems \ref{ramified-computation} and \ref{unramified-computation} that 
\[\ell(W_{\widetilde{\pi}}^0)=L^*(1,(\widetilde{\pi})_u, As).\]
However, note that according to Corollary \ref{obvious}, we have 
\[(\widetilde{\pi})_u=\pi_u,\] and hence 
\[\ell^\prime(W_\pi^0)=L^*(1,(\widetilde{\pi})_u, As)=L^*(1,\pi_u, As)= \ell(W_\pi^0)\neq 0,\] 
so that $\ell^\prime=\ell$, since we know that the constant $c(\pi)$ depends only on $\pi$. 

If $E/F$ is ramified, recall that by Lemma \ref{evenconductor}, the conductor of $\pi$ is even, say $m=2l$. We set 
\[q_l=diag(\w_F^l,\dots,\w_F^l,1)\in A_n(F).\]
As the essential vector 
$W_{\widetilde{\pi}}^0$ is right $G_{n-1}(\o_E)$ invariant, we deduce that 
\[\widetilde{\pi}(p_m)W_{\widetilde{\pi}}^0=\widetilde{\pi}(q_l)W_{\widetilde{\pi}}^0.\]
Since $q_l \in A_n(F)$, we can do a change of variable exactly as in (\ref{eq1}), and the rest of the proof in this case then follows verbatim the proof in the unramified case.
\end{proof}

\begin{remark}\label{rmk3}
As mentioned in Remark \ref{rmk2}, proportionality constants between two specific linear forms do arise naturally in several different situations. It seems most natural to determine the value of such a constant by evaluating the forms on a suitable test vector as we do in the proof of Theorem \ref{constant}. We believe that the computations that we do here will have applications in similar other contexts as well.  
\end{remark}

\section{The non-unitary generic case}\label{sec-nonunitary}

It has been quite standard to study distinction for the pair $(GL(n,E),GL(n,F))$ under the assumption that the irreducible admissible representation of $GL_n(E)$ under consideration is unitarizable. This is because the distinguishing linear form for this pair, up to a scalar multiple, is known to be given by integration on the Kirillov model, and the integral converges only under the unitarity hypothesis \cite[Lemma 4 and Proposition 4 (ii)]{fli88}.  

The point of this section is to extend all of our results till now in this paper for irreducible unitary generic representations to irreducible generic representations which may be non-unitary. The key to making this extension is the following simple observation which applies to $G_n(F)$-distinguished representations of $G_n(E)$.

\begin{proposition}\label{regular}
Let $\pi$ be an irreducible generic representation of $G_n(E)$ which is conjugate self-dual; i.e., $\widetilde{\pi} \simeq \pi^\sigma$. Then, $L(s,\pi,As)$ is holomorphic at $s=1$.
\end{proposition}

\begin{proof}
Write $\pi$ as a commutative product of unlinked segments $\d_1\times \dots \times \d_t$. According to Theorem \ref{multiplicativity-standard}, if $L(s,\pi,As)$ has a pole at $s=1$, then either $L(s,\d_i,As)$ has a pole at $s=1$ for some $i$, or $L(s,\d_j \times \d_k^\sigma)$ has a pole at $s=1$ 
for some $(j,k)$. 

We are going to show  that when $\pi$ is conjugate self-dual, the above observation would imply that some segments among the $\d_k$'s are linked, which will then contradict our assumption that $\pi$ is an irreducible generic representation. 

First suppose that $L(s,\d_i,As)$ has a pole at $s=1$. Then by Lemma \ref{poleB}, one has $\d_i\prec\widetilde{\d_i}^{\sigma}$, but by our assumption that $\pi$ is conjugate self-dual, we know that $\widetilde{\d_i}^{\sigma}$ is a $\d_{i^\prime}$, which contradicts the fact that the $\d_i$'s are unlinked. 

We obtain exactly the same contradiction if $L(s,\d_j \times \d_k^\sigma)$ has a pole at $s=1$, appealing to Lemma \ref{poleA} in this case, instead of Lemma \ref{poleB}.
\end{proof}

Once again let $\psi$ be a non-trivial character of $E$ that is trivial on $F$. Recall that by Lemma \ref{inclusion}, all the integrals 
\[I_{(0)}(s,W)=\int_{N_n(F)\backslash P_n(F)} W(h)\nu_F(h)^{s-1} dh\]
belong to the fractional ideal $I(\pi)$ of $\C[q_F^{\pm s}]$, and in fact they are of the form $I(s,W,\Phi)$ for suitably chosen Schwartz-Bruhat functions $\Phi$.
And we also have just observed in Proposition \ref{regular} that $L(s,\pi,As)$ is holomorphic at $s=1$ since $\pi$ is assumed to be distinguished with respect to $G_n(F)$. Therefore, as a consequence it follows that all the integrals $I_{(0)}(s,W)$ are holomorphic at $s=1$ when $\pi$ is distinguished. 

Hence we can still define the $P_n(F)$-invariant linear form 
\[\ell(W)=I_{(0)}(1,W)\]
on $W(\pi,\psi)$. By Proposition \ref{P-invariance}, it is also $G_n(F)$-invariant.
Similarly, we can also define another $G_n(F)$-invariant linear form 
\[\ell^\prime(W)=I_{(0)}(1,\widetilde{W}).\]

If $\pi$ is ramified, exactly as in Theorem \ref{ramified-computation}, one shows that if $\psi$ has conductor zero, then 
\[I_{(0)}(s,W_\pi^0)=L(s,\pi_u, As)\] for $\mathrm{Re}(s)$ large enough, and hence for all $s \in \mathbb C$. If $\pi$ is unramified, the proof of Theorem \ref{unramified-computation} shows that
\[I_{(0)}(s,W_\pi^0)=\frac{L(s,\pi, As)}{L(ns,1_{F^\times})},\] 
for $s \in \mathbb C$.

Evaluating at $s=1$, we get Theorems \ref{ramified-computation} and \ref{unramified-computation} for any irreducible generic representation of $G_n(E)$ that is distinguished with respect to $G_n(F)$. The proof of Theorem \ref{constant} also goes through without any further modification, except the usual and standard arguments of analytic continuation of invariant linear forms. 

Thus, we obtain the following theorem.

\begin{theorem}\label{general}
If $\pi$ is an irreducible  generic representation of $G_n(E)$ which is distinguished with respect to $G_n(F)$, then:
\begin{enumerate}
\item If $\psi$ has conductor zero, then $\ell(W_\pi^0)=L^*(1,\pi,As)\neq 0$.
\item Moreover, $\ell^\prime =\ell$.
\end{enumerate}
\end{theorem}

\begin{remark}\label{rmk-dp}
There is another striking way of phrasing Theorem \ref{general} (2), which we do now, and we thank Dipendra Prasad for suggesting this remark. Let $\pi$ be an irreducible generic representation of $G_n(E)$ which is distinguished with respect to $G_n(F)$. By Proposition \ref{basic}, we know that 
\[\pi \cong \pi^{\vee \sigma} \cong \pi^\iota,\]
where $\iota:G_n(E) \rightarrow G_n(E)$ is the involution given by $\iota(g)=w ~{^t}g^{-\sigma}~ w^{-1}$. Let $\psi$ be a non-trivial additive character of $E/F$, and let $\lambda: \pi \rightarrow \mathbb C$ be a $\psi$-Whittaker functional (which is unique up to multiplication by scalars). Let $T_i: \pi \rightarrow \pi$ be the \textit{unique} linear map such that
\begin{equation}\label{eq-1}
\lambda \circ T_\iota = \lambda, 
\end{equation}
with
\begin{equation}\label{eq-2}
T_i(g \cdot v) = \iota(g) T_i(v). 
\end{equation}
Indeed, from (\ref{eq-2}), we may assume $T_i^2=1$, and since the involution $\iota$ preserves $N(E)$ and $\psi^\iota = \psi^{-\sigma} = \psi$, it follows that $\lambda \circ T_\iota$ is also a $\psi$-Whittaker functional, thus equals $\pm \lambda$, and therefore (\ref{eq-1}) gives a \textit{canonical} choice of $T_i$. We extend the representation $\pi$ of $G_n(E)$ to the semi-direct product $G_n(E) \rtimes \mathbb Z/2$, where $\mathbb Z/2$ acts via $\iota$, by prescribing 
\begin{equation}\label{eq-3}
(g,\iota) \cdot v = g \cdot T_\iota(v). 
\end{equation}
Now, observe that the linear map
\[W \mapsto \pi(w)\widetilde{W}^\sigma, \]
defined from the Whittaker model $W(\pi,\psi)$ to itself, satisfies both (\ref{eq-1}) and (\ref{eq-2}), and so by the unicity of $T_i$, this map is indeed $T_i$. Therefore, Theorem \ref{general} (2) is equivalent to the identity: 
\begin{equation}\label{eq-4}
\ell \circ T_\iota = \ell.
\end{equation}
From (\ref{eq-3}) and (\ref{eq-4}), we conclude that the $G_n(F)$-invariant linear form $\ell$ is in fact invariant under the larger group $G_n(F) \rtimes \mathbb Z/2$.
\end{remark}

\section*{Acknowledgements}

The authors would like to thank Dipendra Prasad for asking the question about an explicit test vector for the invariant linear form for $(G_n(E),G_n(F))$. They would also like to thank Omer Offen for useful conversations on the theme of this paper. This paper began when the authors were visiting CIRM Luminy as part of the Chaire Jean-Morlet 2016 Programme and they thank Dipendra Prasad and Volker Heiermann for the invitation to participate in the programme. Thanks are due to the anonymous referee for a careful reading of the manuscript and several useful suggestions. The second named author also thanks the grant ANR-13-BS01-0012
FERPLAY for financial support.

\end{document}